\documentclass[a4paper,11pt,DIV=11,
abstract=on
]{scrartcl}
\usepackage{mathtools}
\mathtoolsset{showonlyrefs}
\usepackage{amsfonts}
\usepackage{amsmath}
\usepackage{amssymb}
\usepackage{faktor}
\usepackage{amscd}
\usepackage{array}
\usepackage{amsthm}
\usepackage{subfigure}
\usepackage{mathdots}
\usepackage{mathrsfs}
\usepackage{algorithm}
\usepackage[noend]{algpseudocode}
\usepackage{graphicx}
\usepackage[shortlabels]{enumitem}
\usepackage{color}
\usepackage{pgf}
\usepackage[left=2.5cm, right=2.5cm,top=2.5cm,bottom=3.5cm]{geometry}
\usepackage{scalerel}
\usepackage{bm}
\usepackage{cite}

\newtheorem{thm}{Theorem}[section]

\newtheorem{example}[thm]{Example}
\newtheorem{corollary}[thm]{Corollary}

\newcommand{\N}{\ensuremath{\mathbb{N}}}

\DeclareMathOperator{\diag}{diag}

\DeclareMathOperator{\id}{id}

\DeclareMathOperator{\var}{var}

\newcommand{\bx}{\bm{x}}

\newcommand{\tT}{\mathrm{T}}

\begin{document}

\title{On the Dynamical System \\
  of Principal Curves in $\mathbb R^d$}
\author{
  Robert Beinert\footnote{Institute of Mathematics,
    Technische Universität Berlin,
    Stra{\ss}e des 17.~Juni~136, 
    10623 Berlin, Germany,
    \{beinert,berdellima,graef,steidl\}@math.tu-berlin.de} 
  \and
  Arian B\"erd\"ellima\footnotemark[1] 
  \and
  Manuel Gr\"af\footnotemark[1]
  \and
  Gabriele Steidl\footnotemark[1]
}
\date{\today}
\maketitle

\begin{abstract}
  Principal curves are natural generalizations of principal lines
  arising as first principal components in the Principal Component
  Analysis.  They can be characterized---from a stochastic point of
  view---as so-called self-consistent curves based on the conditional
  expectation and---from the variational-calculus point of view---as
  saddle points of the expected difference of a random variable and
  its projection onto some curve, where the current curve acts as
  argument of the energy functional.  Beyond that, Duchamp and
  St\"utzle (1993,1996) showed that planar curves can by computed as
  solutions of a system of ordinary differential equations.  The aim
  of this paper is to generalize this characterization of principal
  curves to $\mathbb R^d$ with $d \ge 3$.  Having derived such a
  dynamical system, we provide several examples for principal curves
  related to uniform distribution on certain domains in $\mathbb R^3$.
\end{abstract}



\section{Introduction}\label{sec:intro}
Principal component analysis (PCA) \cite{P1901} is still the working
horse of dimensionality reduction algorithms.  The dimensionality
reduction of data contained in $\mathbb R^d$ is here realized by
projecting them onto the low-dimensional affine subspace that
minimizes the sum of the squared Euclidean distances between the data
points and their orthogonal projections.  If the affine subspace is
one-dimensional, PCA just finds a principal line.  Considering the
data as realization of a random variable
$\bm X:\Omega \to \mathbb R^d$, we may compute the principal line as
minimizer of
\begin{equation}
  \mathbb E \left[\|\bm X - \pi_g (\bm X)\|^2 \right]
\end{equation}  
over all lines $g$ in $\mathbb R^d$, where $\pi_g$ denotes the
orthogonal projection onto $g$.  Throughout this paper $\| \cdot\|$
denotes the Euclidean norm.  There are many attempts to generalize
principal lines in the literature.  One direction is to replace the
linear space $\mathbb R^d$ by a nonlinear space.  For instance, if
$\mathbb R^d$ is substituted by a Riemannian manifold, we may ask for the
geodesic that minimizes
\begin{equation}
  \mathbb E \left[\mathrm{d} (\bm X , \pi_g( \bm X))^2 \right]
\end{equation}
over all geodesics $g$, where $\mathrm{d}(\cdot,\cdot)$ denotes the
distance on the Riemannian manifold.  Among the large amount of
literature about PCA on manifolds, we refer to
\cite{fletcher2013,HHM2010,P2015,SLN2013} and the references therein.

Another generalization keeps the linear space setting but asks for a
smooth curve $\Gamma: [0,\ell] \rightarrow \mathbb R^d$ that is a
critical point of
\begin{equation} \label{g3}
  \mathbb E \left[\|\bm X - \pi_\Gamma (\bm X)\|^2 \right].
\end{equation}
These curves---called principal curves by Hastie \cite{Has84} and
Hastie \& St\"utzle \cite{HS89}---possess the so-called
self-consistency property, which can be explained via conditional
expectations.  For principal curves in the plane, Duchamp \& St\"utzle
\cite{DS1996} prove that these are indeed saddle points of \eqref{g3}.
This is quite contrary to the behaviour of principal lines, which are
local minima.  Moreover, in the companion paper \cite{DS93}, Duchamp
\& St\"utzle show that planar principal curves are solutions of a
system of ordinary differential equation.  By solving this dynamical
system, Duchamp \& St\"utzle find principal curves for uniform
densities on rectangles and annuli.

From a numerical point of view, there are several papers on efficient
computations of principal-like curves for point clouds, which can seen
as finitely many samples with respect to the random variable $\bm X$.
Usually, these proposed algorithms require additional constraints on
the curve \cite{CYYLP2016,Kegl1999,KKLZ2000}.  On the basis of these
algorithms, principal curves have found applications in image
processing like the ice floe identification in satellite images
in \cite{BR1992} or like the feature extraction and classification in
\cite{CG1998a}, speech recognition \cite{RN1998}, and engineering
problems \cite{DM1996}.  A more recent generalization of principal
curves to manifolds was considered in \cite{Hauberg2016}, and principal
curves on spheres were discussed in \cite{KLO2020}.

The aim of this paper is to generalize the characterization of
principal curves by differential equations to $\mathbb R^d$,
$d \ge 3$.  Based on our findings, we will compute principal curves
for uniform distributions on specific domains.  When finishing this
paper, we realized that an ingredient of our computation---namely the
generalization of the so-called transverse moments to
$\mathbb R^d$---has been mentioned in \cite{Del01}, however, without
relating the generalized moments to differential equations.  Finally,
we like to mention an other, completely different, powerful method to
approximate arbitrary measures by measures supported on curves based
on the minimization of the Wasserstein distance
\cite{GKL2019,LGKW2018} or the discrepancy \cite{EGNSS2021} between
such measures.

This paper is organized as follows. In Section~\ref{sec:prelim}, we
provide necessary preliminaries on probability theory. Then, in
Section~\ref{sec:pc}, we recall the definition of principle curves
from the stochastic as well as from the variational point of view.
The characterization of principal curves by a system of differential
equations is derived in Section~\ref{sec:ds}.  We apply our findings
for computing principal curves with respect to uniform distribution on
several domains in $\mathbb R^3$ in Section~\ref{sec:num}.  Finally,
we draw conclusions in Section~\ref{sec:conc}.

\section{Preliminaries in Probability Theory} \label{sec:prelim}
%
In the following, we introduce the necessary notation from probability
theory \cite{Kle20}.  Let $(\Omega, \mathcal A, P)$ be a
\emph{probability space}.  By $\mathcal B(\mathbb R^d)$ we denote the
Borel-$\sigma$-algebra on $\mathbb R^d$.  For a random variable
$\bm X = (X_1,\ldots,X_d) : \Omega \to \mathbb R^d$, the
\emph{push-forward measure} $P_{\bm X} : \mathbb R^d \to [0,1]$ of $P$
by $\bm X$ given by
\[
  P_{\bm X}(A) := P( \bm X^{-1}(A)), \qquad A \in \mathcal B(\mathbb R^d),  
\] 
is called the \emph{distribution} of $\bm X$.  We write
$\bm X \sim P_{\bm X}$.  A random variable
$\bm X:\Omega \to \mathbb R^d$ on a probability space
$(\Omega, \mathcal A, P)$, is called \emph{integrable} if
$\bm X \in L_1(\Omega, P)$, i.e.\
$\int_\Omega |\bm X(\omega)| \mathrm d P(\omega) < \infty$.  If
$\bm X$ is integrable, then the \emph{expectation} of $\bm X$ is
defined by
\[
  \mathbb E[\bm X] := (\mathbb E[X_1],\dots, \mathbb E[X_d]), \quad 
  \mathbb E[X_i] := \int_\Omega X_i (\omega) \mathrm d P(\omega).
\]
If $\bm X$ is \emph{square-integrable}, i.e.\ $\bm X \in L_2(\Omega,P)$, then
the \emph{covariance matrix} is defined as 
\[
  \mathrm{Cov}[\bm X] := \big( \mathrm{Cov}[X_i, X_j] \big)_{i,j=1}^d  
  = 
  \big( \mathbb E[ (X_i - \mathbb E[X_i]) (X_j - \mathbb E[X_j])] \big)_{i,j=1}^d 
  \in \mathbb R^{d\times d}	.
\]

The following theorem, which is a straight-forward generalization of
\cite[Thm~8.12]{Kle20} from $\mathbb R$ to $\mathbb R^d$, verifies the
definition of the conditional expectation of a random variable.

\begin{thm}
  \label{the:condExpectation}
  Let $(\Omega, \mathcal A, P)$ be a probability space, and let
  $\bm X:\Omega \to \mathbb R^d$ be a random vector with
  $\bm X \in L_1(\Omega,P)$. For any sub-$\sigma$-algebra
  $\mathcal F \subset \mathcal A$, there exists a random variable
  $\bm Z:\Omega \to \mathbb R^d$ with the following properties:
  \begin{enumerate}[\upshape(i)]
  \item $\bm Z$ is $\mathcal F$-measurable, i.e.,
    $\bm Z^{-1}(B) \in \mathcal F$ for any
    $B \in \mathcal B(\mathbb R^d)$, and
  \item the expectations are equal on $\mathcal F$, i.e., for all
    $F \in \mathcal F$ holds
    $$\int_F \bm X(\omega) \mathrm d P(\omega) = \int_F \bm Z(\omega) \mathrm d P(\omega).$$
  \end{enumerate}    
  If $\tilde{\bm Z}:\Omega \to \mathbb R^d$ is another random vector
  satisfying {\upshape(i)} and {\upshape(ii)}, then
  \[
    P(F_*) = 0, \quad F_* := \{\omega \;:\; \bm Z(\omega) \ne \tilde{\bm Z}(\omega)\} \in \mathcal F.
  \]
  In particular, $\bm Z$ is uniquely determined almost everywhere
  (with respect to the measure $P|_\mathcal F$).
\end{thm}

The random vector $\bm Z$ is called the \emph{conditional expectation
  of $\bm X$ given $\mathcal F$}, and we use the notation
$\mathbb E[\bm X | \mathcal F] := \bm Z$.  For
$\bm X:\Omega \to \mathbb R^d$, and for a random variable
$\bm Y:\Omega \to \mathbb R^{p}$, we define the \emph{conditional
  expectation of $\bm X$ given $\bm Y$} by
$\mathbb E[\bm X | \bm Y]:= \mathbb E[\bm X | \sigma(\bm Y)] : \Omega
\to \mathbb R^{p}$, where $\sigma(\bm Y)$ denotes the smallest
$\sigma$-algebra containing the set system
$\{ \bm Y^{-1}(B) \;:\; B \in \mathcal B(\mathbb R^{p})\})$.  By the
factorization lemma \cite[Cor~1.97]{Kle20}, there exists a measurable
function $\varphi:\mathbb R^p \to \mathbb R^d$ such that
\begin{equation}     \label{eq:cond-exp}
  \mathbb E[\bm X|\bm Y] (\omega)= \varphi(\bm Y (\omega)), \qquad \omega \in \Omega.
\end{equation}
We call $\varphi$ the \emph{conditional expectation of $\bm X$ given
  $\bm Y(\omega) = \bm y$} and use the notation
$\mathbb E[\bm X | \bm Y = \bm y] := \varphi(\bm y)$.  Denoting by
$\mathbb E_{\bm Y}$ the expectation with respect to probability space
$(\mathbb R^p, \mathcal B(\mathbb R^p), P_{\bm Y})$, where $P_{\bm Y}$
is the push-forward measure of $\bm Y$, we observe
\[
  \mathbb E[\bm X] = \int_{\mathbb R^p} \mathbb E[\bm X| \bm Y=\bm y] \, \mathrm d P_{\bm Y}(\bm y) 
  = \mathbb E_{\bm Y} \left[\mathbb E[\bm X | \bm Y = \cdot] \right].
\]

\section{Principal Curves} \label{sec:pc}
Throughout this paper, we consider smooth Jordan curves
$\Gamma:[0,\ell_{\Gamma}] \to \mathbb R^d$ parameterized by their
arc-length.  This means that $\Gamma \in C^\infty([0,\ell_\Gamma])$
does not intersect itself, i.e., $s_1 \ne s_2$ implies
$\Gamma(s_1) \ne \Gamma(s_2)$.  The distance to the curve is given by
$d(\bx,\Gamma):=\min_{\bm y\in\Gamma}\|\bx - \bm y\|$, where the
minimum is realized at least once since $\Gamma$ is compact.  If $\bx$
has several such closest points on $\Gamma$, then $\bx$ is said to be
an \emph{ambiguity point}.  The set of ambiguity points $\mathcal A_\Gamma$
is of Lebesgue measure zero, see \cite[Lem~4.3.2]{Has84} and
\cite[Prop~6]{HS89}.  The \emph{projection index}
$\lambda_\Gamma : \mathbb R^d \to [0,\ell_\Gamma]$ was introduced by
Hastie \cite{Has84} as
\[
  \lambda_\Gamma(\bx) := \sup\left\{ s \in [0,\ell_\Gamma] \;:\; \| \bx - \Gamma(s) \| 
    = \inf_{t \in [0,\ell_\Gamma]}\| \bx - \Gamma(t) \|\right\}, \qquad \bx \in \mathbb R^d.
\]
Based on the projection index, we define the \emph{projection}
$\pi_\Gamma : \mathbb R^d \to \Gamma$ as composition
$\bm x \mapsto \Gamma \circ \lambda_\Gamma (\bm x)$. By slight abuse
of notation we identify $\Gamma$ with its image here.  Note that the
projection is always singe-valued even for the ambiguity points.
Hastie \cite[Thm~4.1]{Has84} has shown that $\lambda_\Gamma$ is
measurable for smooth curves.  Hence, for a random variable
$\bm X: \Omega \rightarrow \mathbb R^d$, the composition
$\lambda_\Gamma \circ \bm X : \Omega \to [0,\ell_\Gamma]$ is also a
random variable as well as $\Gamma(\lambda_\Gamma(\bm X))$, and we have
$\pi_\Gamma (\bm X) = \Gamma(\lambda_\Gamma(\bm X))$.  By the
factorization in \eqref{eq:cond-exp}, we can write the conditional
expectation as
\begin{equation}     \label{eq:exp-proj}
  \mathbb E[\bm X | \lambda_\Gamma \circ \bm X](\omega) 
  = \phi \circ \lambda_\Gamma \circ \bm X (\omega),
\end{equation}
with $\mathbb E[ \bm X | \lambda_\Gamma(\bm X) = s] = \phi(s)$.  A
curve $\Gamma$ is called \emph{self-consistent} if and only if
\[
  \mathbb E[ \bm X | \lambda_\Gamma(\bm X) = s] = \Gamma (s) \qquad P_{\lambda_\Gamma(\bm X)}-\text{a.e.}
\]
for all $s \in [0, \ell_\Gamma]$.  A smooth, self-consistent Jordan
curve is called a \emph{principal curve} of $\bm X$ \cite{HS89}.

For uniformly distributed random variables
$\bm X: \Omega \rightarrow \mathbb R^2$, the definition says that a
principal curve is characterized by the fact that the barycenter of
the region $\lambda^{-1}_\Gamma(I)$ related to some interval
$I \subset [0,\ell_\Gamma]$ converges to $\Gamma$ if the length of $I$
becomes arbitrary small.  Numerically, some example regions
$\lambda^{-1}_\Gamma(I)$ may be calculated using the Voronoi cells
with respect to finitely many samples on $\Gamma$, which allow a
numerical validation whether a curve is principal for a given uniform
distribution.  This definition and numerical
interpretation is illustrated in Figure \ref{fig:1}.

\begin{figure} \label{fig:1}
  \begin{center}
    \includegraphics[width=0.35\textwidth]{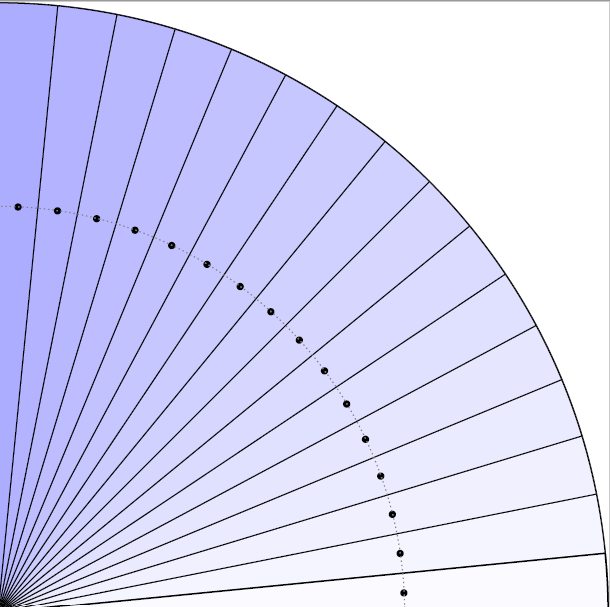}
    \hspace{1cm}
    \includegraphics[width=0.35\textwidth]{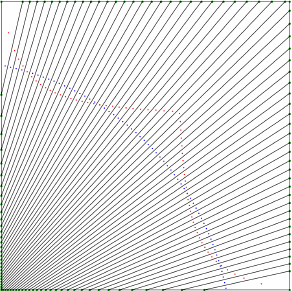}
  \end{center}
  \caption{ Two examples for curves and uniformly distributed random
    variable $\bm X$.  Left: $\bm X$ corresponds to the uniform
    distribution on the quarter circle of radius 1. Its principal
    curve is just the quarter circle of radius $\frac23$.  The
    barycenters of the Voronoi cells converge to the curve if the
    curve is sampled denser.  Right: $\bm X$ corresponds to the
    uniform distribution on the square of side length 1.  For a
    quarter parabolic curve (blue dots), we calculate the barycenters
    (red dots) of the corresponding Voronoi cells.  These centers do
    not converge to the curve even if the regions become arbitrary
    thin; so the shown curve is not principal.}
\end{figure}

Principal curves have a nice variational characterization. 
To this end, we consider the energy functional 
\begin{equation}
  \label{eq:PC_energy}
  D_{\bm X}^2(\Gamma) := \mathbb E[ \| \bm X - \Gamma (\lambda_\Gamma(\bm X)) \|^2],
\end{equation}
whose critical points are principal curves.

\begin{thm}[Hastie \& Stuetzle {\cite[Prop~4]{HS89}}]
  Let $\bm X: \Omega \rightarrow \mathbb R^d$ be a random variable
  with finite covariance and smooth density.  Further, let
  $\Gamma :[0,\ell_\Gamma]\to \mathbb R^d$ be a smoooth Jordan curve
  parameterized by arc-length.  Then the curve $\Gamma$ is a principal
  curve of $\bm X$ if and only if
  \[
    \frac{\mathrm d D_{\bm X}^2(\Gamma + t \Delta)}{\mathrm d t} \Big|_{t=0} = 0,
  \]
  for any curve $\Delta \in C^{\infty}([0,\ell_\Gamma])$ with
  $\| \Delta \| \le 1$, $\| \Delta' \| \le 1$.
\end{thm}

\section{Characterization via Differential Equations} \label{sec:ds}
%
A third characterization of principal curves in the plane is given by
a dynamical system \cite{DS93}.  In this section, we generalize the
derivation to curves $\Gamma: [0,\ell_\Gamma] \rightarrow \mathbb R^d$
in higher dimensions, i.e.\ $d\ge 2$.  For this, we associate to
$\Gamma(s)$ a reference frame
$T(s), N_1(s), \dots, N_{d-1}(s) \in \mathbb S^{d-1}$ smoothly
depending on $s$, where $T(s) := \Gamma'(s) \in \mathbb S^{d-1}$
denotes the tangent, and where $N_1(s), \dots, N_{d-1}(s)$ are
pairwise orthogonal vectors spanning the normal space of $\Gamma$ in
$s \in [0,\ell_\Gamma]$.  Recall that the curvature of a curve is
given by $\kappa(s):=\|T'(s)\|$.  The \emph{principal curvatures} with
respect to the chosen moving reference frame are now defined by
\begin{equation} \label{eq:kappa}
  \kappa_i(s) :=\langle T\,'(s),N_i(s) \rangle, \qquad i=1, \dots, d-1.
\end{equation}
In other words, the principal curvatures
$(\kappa_1,\dots, \kappa_{d-1})$ are the coordinates of the normal
$\Gamma'' = T'$ with respect to the frame $N_1, \dots, N_{d-1}$.  Due
to the orthogonality $\langle T, N_i \rangle = 0$, we have
\[
  0 = \frac{d}{ds} \langle T(s), N_i(s) \rangle = \langle T'(s),
  N_i(s) \rangle + \langle T(s), N_i'(s) \rangle,
  \qquad i=1, \dots, d-1,
\]
implying
\begin{equation}
  \label{eq:kappa_TN'}
  \kappa_i(s) = - \langle T, N_i'(s) \rangle, \qquad i=1,\dots, d-1.
\end{equation}

There are different kind of frames in the literature, e.g., the Frenet
frame, the Bishop frame, and various modifications
\cite{Bishop1975,YILMAZ2010764}.  The Frenet frame is unique, but may
fail to be well defined at certain points even if the
curve is sufficiently regular.  In contrast, the Bishop frame---also
known as parallel frame---is defined at every point and
varies continuously as we move along the curve.  This frame is described by
the system of first order differential equations
\begin{align}
  \label{eq:Bishop2}
  \begin{bmatrix}
    T'(s)\\N_1'(s)\\N_2'(s)\\\vdots\\N_{d-1}'(s)
  \end{bmatrix}=\begin{bmatrix}
    0& \kappa_1(s)&\kappa_2(s)&\cdots&\kappa_{d-1}(s)\\
    -\kappa_1(s)&0&0&\cdots&0\\
    -\kappa_2(s)&0&0&\cdots&0\\
    \vdots & \vdots&\vdots&\ddots&\vdots \\
    -\kappa_{d-1}(s) &0&0&\cdots &0
  \end{bmatrix}\cdot \begin{bmatrix}
    T(s)\\N_1(s)\\N_2(s)\\\vdots\\N_{d-1}(s)
  \end{bmatrix}.
\end{align}
In the numerical part, we will rely on a different frame based on spherical coordinates.

Henceforth, let $\mathbb X$ be a compact region in $\mathbb R^d$, which will
later denote the support of the density $p_{\bm X}$.  The
\emph{normal coordinate map of} $\Gamma$ with respect to the chosen reference
frame is the map
$\nu_\Gamma:[0, \ell_\Gamma]\times\mathbb R^{d-1}\to\mathbb R^d$ given
by
\begin{equation} 	\label{eq:normal-map}
  \nu_{\Gamma}(s,u_1,\dots, u_{d-1})\coloneqq \Gamma(s) + u_1\,N_1(s) + \cdots u_{d-1} \,N_{d-1}(s),
\end{equation}
and the \emph{normal coordinate transformation}
$\mu_{\Gamma}: \mathbb X \to [0,\ell_\Gamma]\times \mathbb R^{d-1}$ is
defined by
\begin{equation} 	\label{eq:normal-transform}
  \mu_{\Gamma}(\bx)\coloneqq 
  \begin{pmatrix}
    \lambda_\Gamma(\bx)
    \\
    \langle \bx- \pi_\Gamma(\bx), N_1(\lambda(\bx))\rangle
    \\
    \vdots
    \\
    \langle  \bx- \pi_\Gamma(\bx), N_{d-1}(\lambda(\bx))\rangle
  \end{pmatrix}.
\end{equation}
The components $(s,u_1,\dots, u_{d-1})$ of $\mu_{\Gamma}(\bx)$ are
called the \emph{normal coordinates at} $\bx$.  For given
$s\in[0,\ell_\Gamma]$, let $\mathbb X(s)$ be the cross-section of
$\mathbb X$ with the hyperplane
$\Gamma(s)+u_1N_1(s)+ \cdots + u_{d-1}N_{d-1}(s)$.  The normal
coordinates around $\Gamma(s)$ in $\mathbb X(s)$ are denoted by
\begin{align*}
\mathcal  R(s)=\{(u_1,\dots,u_{d-1})\in\mathbb R^{d-1}\;:\;
  &\pi_\Gamma(\Gamma(s)+u_1 N_1(s)+ \cdots + u_{d-1}N_{d-1}(s)) = s,\\
  & \Gamma(s)+u_1 N_1(s)+ \cdots + u_{d-1}N_{d-1}(s) \not \in
    \mathcal A_\Gamma\}.
\end{align*}

For the later substitution, we need that $\nu_\Gamma$ is a
diffeomorphism on $(\lambda_\Gamma^{-1}(I) \cap \mathbb X) \setminus
\mathcal A_\Gamma$ for all
measurable $I \subseteq [0, \ell_\Gamma]$ meaning that
$$(\nu_\Gamma
\circ \mu_\Gamma)|_{(\lambda_\Gamma^{-1}(I) \cap \mathbb X) \setminus
  \mathcal A_\Gamma} = \id|_{(\lambda_\Gamma^{-1}(I) \cap \mathbb X)
  \setminus \mathcal A_\Gamma},$$ and that $\nu_\Gamma$ and $\mu_\Gamma$ are
differentiable on the related domains.  The partial derivatives of $\nu_\Gamma$ are
given by
\begin{equation*}
  \begin{aligned}
    \frac{\partial \nu_{\Gamma}}{\partial s} & =\Gamma'(s)+u_1\,N'_1(s)+ \cdots u_{d-1}\,N'_{d-1}(s)
    ,\\
    \frac{\partial \nu_{\Gamma}}{\partial u_i} & =N_i(s),\qquad i=1,\dots,d-1.
  \end{aligned}
\end{equation*}
Using \eqref{eq:kappa_TN'}, and exploiting the orthonormality of the
frame, we obtain the Jacobian determinant 
\begin{align*}
  |\det(\bm J_{\nu_\Gamma})|(s,\bm u)
  &= |T(s) + u_1 \, N_1'(s) + \cdots + u_{d-1} \, N_{d-1}'(s), N_1(s),\dots, N_{d-1}(s)|
  \\
  &= |\langle T(s), T(s) + u_1 \, N_1'(s) + \cdots + u_{d-1} \, N_{d-1}'(s) \rangle T(s), N_1(s),\dots, N_{d-1}(s)|\\
  &= 1 + u_1 \langle T(s), N_1'(s) \rangle + \cdots + u_{d-1} \langle T(s), N_{d-1}'(s) \rangle\\
  & = 1- u_1 \kappa_1 - \cdots - u_{d-1} \kappa_{d-1}. 
\end{align*}
Now we can describe the self-consistency of curves with respect to a
random variable based on its \emph{transverse moments}
\begin{equation} \label{eq:trans-mom}
  \mu_{\bm j}(s) :=
  \int\limits_{\bm u \in \mathcal  R(s)} u_{1}^{j_1} \cdots u_{d-1}^{j_{d-1}} 
  \, p_{\bm X}\Big(\Gamma (s)+\sum_{i=1}^{d-1} u_{i}\,N_{i}(s) \Big) \, d \bm u, \qquad \bm j \in \mathbb N_0^{d-1}.
\end{equation}
Further, the canonical basis of $\mathbb R^{d-1}$ is denoted by
$\bm e_1, \dots, \bm e_{d-1} \in \mathbb R^{d-1}$.

\begin{thm} \label{theo:moments}
  Let $\bm X: \Omega \rightarrow \mathbb R^d$ 
  be a random variable having a  distribution 
  with smooth density function $p_{\bm X}$, 
  where $\mathrm{supp} \, p_{\bm X} = \mathbb X$
  and $p_{\bm X}$ 
  is strictly positive in the interior of $\mathbb X$.
  We consider smooth Jordan curves $\Gamma: [0,\ell_{\Gamma}] \rightarrow \mathbb X$ for which 
  $\nu_\Gamma$ is a diffeomorphism on 
  $(\lambda_\Gamma^{-1}([0,\ell_{\Gamma}]) \cap \mathbb X) \setminus \mathcal A_\Gamma$.
  Then $\Gamma$ is self-consistent with respect $\bm X$
  if and only if its principal curvature $\mathbf{\kappa}$
  fulfills the linear system of equations
  \begin{equation} \label{eq:curvature_condition}
    \bm \mu (s) = \bm G(s) \big(\kappa_1(s), \dots, \kappa_{d-1}(s)\big)^\tT
  \end{equation}
  with
  \begin{equation*}
    \bm G(s):=\big( \mu_{\bm e_i + \bm e_j}(s) \big)_{i,j=1}^{d-1}, \quad 
    \bm \mu (s) := \big( \mu_{\bm e_i}(s) \big)_{i=1}^{d-1}.
  \end{equation*}
  The Gram matrix $\bm G(s)$ is invertible, so that
  \begin{equation}
    \label{eq:cur-sys}
    \kappa_j=(\bm G^{-1} \bm\mu)_j, \qquad j = 1,\dots, d-1.
  \end{equation}
\end{thm}

\begin{proof}
  Based on Theorem~\ref{the:condExpectation} and the factorization of
  the conditional expectation in \eqref{eq:exp-proj}, for all
  measurable sets $I \subseteq [0,\ell_\Gamma]$, the self-consistency
  $\phi(s) = \mathbb E[\bm X|\lambda_{\Gamma}(\bm X)=s] = \Gamma(s)$
  means
  \begin{align*}
    \smashoperator[r]{\int_{(\lambda_\Gamma \circ \bm X)^{-1}(I)}} \bm X (\omega) \, dP(\omega)
    &= \smashoperator{\int_{(\lambda_\Gamma \circ \bm X)^{-1}(I)}} \mathbb E[\bm X | \lambda_\Gamma \circ \bm X](\omega) \, dP(\omega)
      = \smashoperator{\int_{(\lambda_\Gamma \circ \bm X)^{-1}(I)}} 
      (\phi \circ \lambda_\Gamma \circ \bm X) (\omega) \, dP(\omega)
    \\
    &= \smashoperator{\int_{(\lambda_\Gamma \circ \bm X)^{-1}(I)}} 
      (\Gamma \circ \lambda_\Gamma \circ \bm X) (\omega) \, dP(\omega).
  \end{align*}
  This can be rewritten as
  \begin{equation}
    \label{eq:integral}
    \smashoperator[r]{\int_{\lambda^{-1}_{\Gamma}(I)}} \bx \, p_{\bm X}(\bx)\,d \bm x
    =\smashoperator{\int_{\lambda^{-1}_{\Gamma}(I)}} (\Gamma \circ \lambda_\Gamma) (\bx) \, p_{\bm X}(\bx)\,d\bx
    =\smashoperator{\int_{\lambda^{-1}_{\Gamma}(I)}} \pi_{\Gamma}(\bx) \, p_{\bm X}(\bx)\,d\bx
  \end{equation}
  or, equivalently,
  \begin{equation}
    \label{eq:integral2}
    \smashoperator[r]{\int_{\lambda^{-1}_{\Gamma}(I)}}(\bx-\pi_{\Gamma}(\bx)) \, p_{\bm X}(\bx)\,d\bx=0
  \end{equation}
  for all measurable sets $I\subseteq [0, \ell_\Gamma]$. 
  Regarding that $\nu_\Gamma$ is a diffeomorphism on
  $(\lambda_\Gamma^{-1}(I) \cap \mathbb X) \setminus \mathcal
  A_\Gamma$, and that $\mathcal A_\Gamma$ is a null set, 
  we can rewrite the last integral condition in terms of the normal coordinates
  \begin{equation*}
    \iint\limits_{(s,\bm u) \in \mu_\Gamma\left(\lambda^{-1}_\Gamma(I)
      \setminus \mathcal A_\Gamma\right)}
    \Big(\sum_{i=1}^{d-1} u_i \, N_i(s) \Big) \,
    p_{\bm X}\Big(\Gamma(s)+\sum_{i=1}^{d-1} u_i \, N_i(s) \Big)
    \, |\det(\bm J_{\nu_\Gamma})| d\bm u \, ds = 0.
  \end{equation*}
  Since the above equation holds for all $I \subseteq [0,\ell_\Gamma]$,  
  the integrand with respect to $s$ thus has to be zero almost surely, i.e. 
  \begin{equation}
    \label{eq:int-cond}
    \int\limits_{\bm u \in \mathcal  R(s)}
    \Big(\sum_{i=1}^{d-1} u_i \, N_i(s) \Big) \,
    p_{\bm X}\Big(\Gamma(s)+\sum_{i=1}^{d-1} u_i \, N_i(s) \Big)
    \, |\det(\bm J_{\nu_\Gamma})| \, d \bm u = 0,     \quad s\text{-a.s.}
  \end{equation}
  Exploiting that the  of $N_i(s)$ are orthonormal 
  and the smoothness of $p_{\bm X}$, we see that the parameter integral is continuous, so that we obtain
  \begin{align}
    \label{eq:integral5}
    & \int\limits_{\bm u \in \mathcal  R(s)} 
      u_j\,p_{\bm X}\Big(\Gamma(s)+\sum_{i=1}^{d-1} u_i\,N_i(s)\Big) \,
      \Big(1-\sum_{i=1}^{d-1} u_i \kappa_i(s) \Big) \, d \bm u=0, \qquad j=1,\dots, d-1
  \end{align}
  for all $s\in [0,\ell_\Gamma]$.  Note that the Jacobian determinant
  is here always positive, since $\mu_\Gamma$ and $\nu_\Gamma$ are
  diffeomorphisms.  Using the transverse moments notation in
  \eqref{eq:trans-mom} this can be rewritten as the system
  \eqref{eq:curvature_condition}.  Since the first-order monomials
  $\bm u \mapsto u_i$, $i=1,\dots, d-1$, are linear independent on
  every open subset in $\mathbb R^{d-1}$, we infer that the Gram
  matrix $\bm G$ is invertible.
\end{proof}

In the following, we fix the moving reference frame by parameterizing
the tangent vector using spherical coordinates
\begin{equation}    \label{eq:sph-cor}
  T(\bm \zeta) := 
  \begin{pmatrix}
    & \cos(\zeta_1) \\
    & \sin(\zeta_1) \cos(\zeta_2) \\
    & \sin(\zeta_1) \sin(\zeta_2) \cos(\zeta_3) \\
    & \vdots \\
    & \sin(\zeta_1) \sin(\zeta_2) \cdots \sin(\zeta_{d-2}) \cos(\zeta_{d-1}) \\
    & \sin(\zeta_1) \sin(\zeta_2) \cdots \sin(\zeta_{d-2}) \sin(\zeta_{d-1}) 
  \end{pmatrix} \in \mathbb S^{d-1}
\end{equation}
where $\zeta_i \in [0, \pi]$, $i = 1, \dots, d-2$ and
$\zeta_{d-1} \in [0,2\pi)$ are functions of $s$.  Note that
$T(\bm \zeta)$ and its partial derivatives
$T_{\zeta_i} := \frac{d}{d\zeta_i} T$ satisfy the recursions
\begin{equation}
  T(\bm \zeta) = 
  \begin{pmatrix}
    \cos(\zeta_1) \\
    \sin(\zeta_1) \, T(\bm \xi)
  \end{pmatrix},
  \qquad 
  T_{\zeta_1}(\bm \zeta) =
  \begin{pmatrix}
    -\sin(\zeta_1) \\
    \cos(\zeta_1) \, T(\bm \xi)
  \end{pmatrix},
  \qquad
  T_{\zeta_k}(\bm \zeta) = \sin(\zeta_1) 
  \begin{pmatrix}
    0 \\
    T_{\xi_{k-1}}(\bm \xi)
  \end{pmatrix},
\end{equation}
where $\bm \xi := (\zeta_2, \dots, \zeta_{d-1})^\tT$.  Consequently,
we have $\|T_{\zeta_1}\| = 1$ and
$\|T_{\zeta_k}\| = \prod_{i=1}^{k-1} \sin(\zeta_i)$, $k=2,\ldots,d-1$.
Defining  the
vectors
\[
  N_i (\mathbf \zeta) := \frac{T_{\zeta_i}(\mathbf \zeta)}{\| T_{\zeta_i}(\mathbf \zeta) \|} \qquad i=1, \dots, d-1,
\]
for $\zeta_k \not \in \{0, \pi\}$, $k=1,\ldots,d-2$, we see that
$\{T(\mathbf{\xi}) ,N_1(\mathbf{\xi}), \ldots,N_{d-1}(\mathbf{\xi})\}$
forms an orthonormal basis of $\mathbb R^d$.  Later we will argue that
the instabilities are not problematic for the numerical part.

Considering the curvature of $\Gamma$ given by
\[
  \frac{d}{ds} T(s) = \sum_{i=1}^{d-1} \zeta_i'(s) T_{\zeta_i}(s),
\]
we conclude from \eqref{eq:kappa} that
\begin{equation} \label{eq:kappa_ds}
  \kappa_i(s) = \left\langle \frac{d}{ds} T(s), N_i(s) \right\rangle = \zeta_i'(s) \|T_{\zeta_i}(s)\|, \qquad i=1,\dots,d-1.
\end{equation}
Inserting these identities into \eqref{eq:cur-sys}, the
self-consistency of a curve with respect to $\mathbf X$ is equivalent
to the system of differential equations
\begin{equation}
  \label{eq:ode_system:1}
  \begin{aligned}
    \Gamma'(s) & = T(\zeta(s)),\\
    \zeta'_j(s) & =  (\bm G^{-1} \bm \mu)_j /  \|T_{\zeta_i(s)}\|
    = (\bm G^{-1} \bm \mu)_j \Bigm\slash  \prod_{k=1}^{i-1} \sin(\zeta_k).
  \end{aligned}
\end{equation}
Note that the moments $\mu_{\bm j}(s)$ are functions depending both on
the point $\Gamma(s)$ and our specific frame characterized by
$\bm \zeta(s)$.  Therefore the right-hand side of the differential
equation system is a function in $(\Gamma(s), \bm \zeta(s))$ and may
be solved using linear multistep methods for instance.

The first and second order transverse moments can be interpreted
stochastically by defining the \emph{transverse density} at time $s$
by
\begin{equation}
  \label{eq:transverse-density}
  p_{\perp}(\bm u, s)\coloneqq\frac{p(\Gamma(s)+u_1 N_1(\bm \zeta(s))+ \cdots + u_{d-1} N_{d-1} (\bm \zeta (s)))}{\mu_{\bm 0}(s)},\quad \bm u\in\mathcal  R(s).
\end{equation} 
The mean and the covariance matrix of the transverse density
$p_{\perp}(\cdot, s)$ with respect to the normal coordinates are given
by
\[
  \begin{aligned}
    v_{\perp}(s) & := \Big( \frac{\mu_{\bm e_1}(s)}{\mu_{\bm 0}(s)}, \dots, \frac{\mu_{\bm e_{d-1}}(s)}{\mu_{\bm 0}(s)} \Big)^\tT,\\
    \mathrm{cov}_{\perp}(s) & := \Big( \frac{\mu_{\bm e_i + \bm e_j}(s)}{\mu_{\bm 0}(s)} \Big)_{i,j=1}^{d-1} - v_{\perp}(s) \, v_{\perp}(s)^\tT.
  \end{aligned} 
\]
Since the 0th transverse moment cancels out, we arrive at the ordinary system of differential equations
\begin{equation} \label{eq:ode_system:2}
  \begin{aligned}
    \Gamma'(s) & = T(\zeta(s))\\
    \bm \zeta'(s) & = D^{-1}(s) \Big( \mathrm{cov}_{\perp}(s) + v_{\perp}(s) v_{\perp}(s)^{\tT} \Big)^{-1} v_{\perp}(s)
  \end{aligned}
\end{equation}
with 
\[
  D(s) := \diag(\|T_{\zeta_1}(s)\|, \dots, \|T_{\zeta_{d-1}}(s)\|).
\]
Our findings are summarized in the following theorem.

\begin{thm}
  Let the assumptions of Theorem \ref{theo:moments} be fulfilled.
  Assume that $\Gamma$ can be represented by the above frame with $\zeta_k(s) \not \in \{0,\pi\}$, $k=1,\ldots, d-2$, 
  $s \in [0,\ell_{\Gamma}]$.
  Then the curve $\Gamma$ is self-consistent with respect to $\bm X$ if and only if it 
  is a solution of the system of differential equations \eqref{eq:ode_system:2}.
\end{thm}

From a numerical point of view the above instabilities causes by the
ambiguousness of the spherical coordinates appear to be
non-problematic.  Notice that the scenery, i.e. the random variable
$\bm X$ with density $p_{\bm X}$, the starting point $\Gamma(0)$, and
the initial tangent $\Gamma'(0)$ may be rotated such that the
spherical coordinates of $\Gamma'(0)$ satisfy
$\zeta_k \not \in \{0,\pi\}$, $k=1,\ldots,d-2$.  If the solution of
\eqref{eq:ode_system:2} is computed step-by-step by a linear multistep
method, we may stop the computation whenever the spherical coordinates
of the tangent become ambiguous.  Rotating the scenery with the
computed curve again, we can continue the computations.

\section{Principal Curves of Uniformly Distributed Random Variables} \label{sec:num}
%
In this section, we are interested in the concrete computation of
principal curves of uniformly distributed random variables with
densities supported at certain specific domains in $\mathbb R^3$.  For
the numerics and the considerations on symmetries, we have to assume
that $\Gamma$ fulfills the following \emph{admissibility assumptions}:
\begin{enumerate}[(i)]
\item $\mathbb X$ contains no ambiguity points with respect to
  $\Gamma$.  This implies that the normal map is the left inverse of
  the normal coordinate map
  \begin{equation} 	\label{eq:left-inverse}
    \nu_{\Gamma}\circ\mu_{\Gamma}=\id_{\mathbb X},
  \end{equation}
\item the map
  $\nu_{\Gamma}:\ [0, \ell_\Gamma]\times\mathbb R^{d-1} \to \mathbb
  R^d$ is a diffeomorphism onto its image.
\end{enumerate}
Excluding any ambiguity points, we are able to compute the transverse
moments $\mu_{\bm j}(s)$ in \eqref{eq:trans-mom} by only knowing the
current position $\Gamma(s)$ and the corresponding tangent
$\Gamma'(s) = T(\bm \zeta(s))$ since the domain of integration
$\mathcal R(s)$ becomes simply the cross-section $\mathbb X(s)$.

\subsection{Symmetric and Rotation-Invariant Domains}
We start with densities having a special symmetric support which will
result in principal curves lying in a plane.  Without loss of
generality, we call the density $p_{\bm X}$ with compact support
$\mathbb X$ \emph{reflectionally symmetric} if
$p_{\bm X}(x_1, \dots, x_{d-1}, x_d) = p_{\bm X}(x_1, \dots, x_{d-1},
- x_d)$.  The hyper-plane $\mathcal H_d$ orthogonal to $\bm e_d$ is
here the \emph{reflection plane}.  Then we have the following theorem.

\begin{thm}
  \label{lem:sym-dom}
  Let $p_{\bm X}$ be reflectionally symmetric. If the admissible
  principal curve $\Gamma$ starts in
  $\partial \mathbb X \cap \mathcal H_d$, then $\Gamma$ remains in
  $\mathcal H_d$.
\end{thm}

\begin{proof}
  Denote by $\mathcal H_d^+ := \{ \bm x : x_d>0\}$ and
  $\mathcal H_d^- := \{ \bm x : x_d<0\}$ the half-spaces with respect
  to $\mathcal H_d$.  Assume $\Gamma(s_1) \in \mathcal H_d$ and
  $\Gamma(t) \in \mathcal H_d^+$ for $t \in (s_1,s_2)$.  Since
  $\Gamma$ is smooth, we may choose $s_2$ such that the hyper-plane
  $\lambda_\Gamma^{-1}(s_2)$ is not reflectionally symmetric with
  respect to $\mathcal H_d$, whereas $\lambda_\Gamma^{-1}(s_1)$ is
  perpendicular to $\mathcal H_d$.  Figuratively, the section
  $\lambda_\Gamma^{-1}(\bar I)$ is squeezed in $\mathcal H_d^+$ and
  stretched in $\mathcal H_d^-$.  Mathematically,
  $\lambda_\Gamma^{-1}(\bar I) \cap \mathcal H_d^-$ has a greater mass
  than $\lambda_\Gamma^{-1}(\bar I) \cap \mathcal H_d^+$.
  Consequently, the conditional mean $\mathbb E[\cdot]$ of
  $\lambda_\Gamma^{-1}(\bar I)$ lies in $\mathcal H_d^-$, whereas the
  conditional mean $\mathbb E[\pi_\Gamma(\cdot)]$ lies in
  $\mathcal H_d^+$.  Thus, the integral \eqref{eq:integral} cannot
  hold true, which contradicts the self-consistency meaning that
  $\Gamma$ cannot leave the hyper-plane $\mathcal H_d$.  The basic
  idea of the proof is schematically shown in
  Figure~\ref{fig:sym}.
\end{proof}

\begin{figure}
  \centering
  \includegraphics[width=0.4\linewidth]{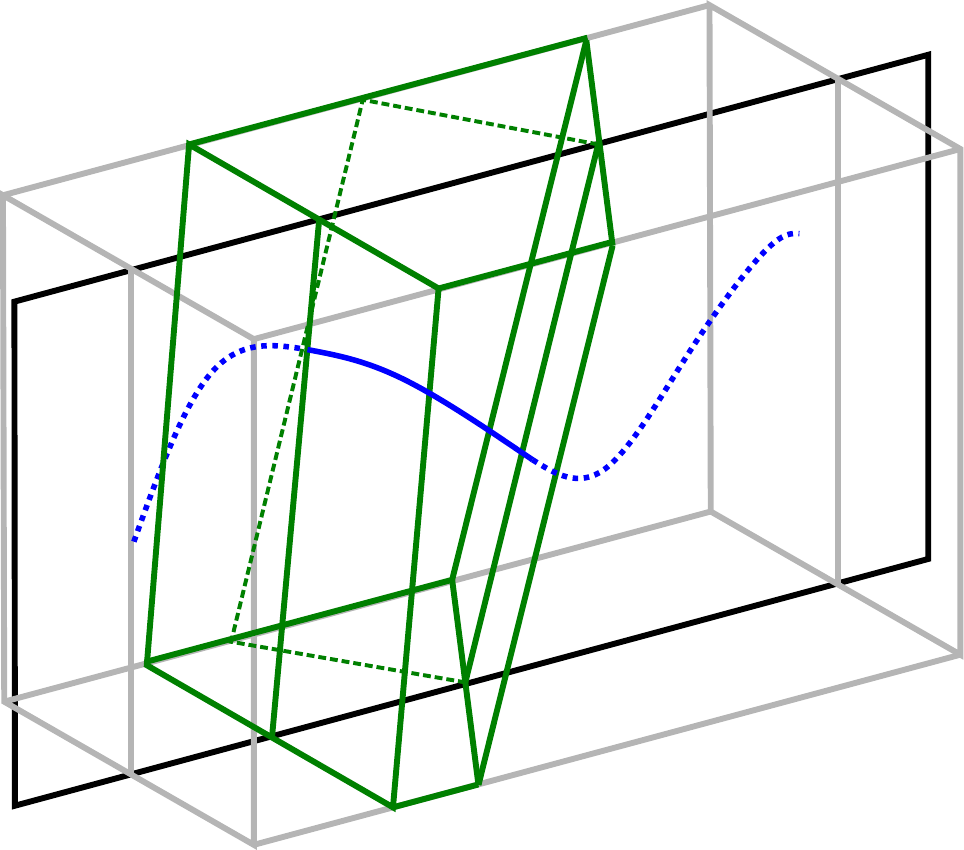}
  \caption{For simplicity, the region $\mathbb X$ (gray) is shown as
    cuboid.  If the curve $\Gamma$ (blue) leaves the plane $\mathcal
    H_3$, the section $\lambda_\Gamma^{-1}(\bar I)$ becomes
    non-symmetric.  The additional
    region in $\mathcal H_3^-$, which is here schematically shown by
    the green, dashed line, pulls the conditional expectation into
    $\mathcal H_3^-$, whereas $\Gamma(I)$ is contained in $\mathcal
    H_3^+$.}
  \label{fig:sym}
\end{figure}

Again without loss of generality, we call the density $p_{\bm X}$ with
compact support $\mathbb X$ \emph{rotationally symmetric} if
$p_{\bm X}(x_1, \dots, x_{d-1}, x_d) = p_{\bm X}(y_1, \dots, y_{d-1},
x_d)$ for $\|x_1,\dots, x_{d-1}\| = \|y_1,\dots, y_{d-1}\|$.

\begin{corollary}
  \label{cor:rot-sym}
  Let $p_{\bm X}$ be rotationally symmetric.  If the admissible
  principal curve $\Gamma$ starts in $\partial \mathbb X$, then
  $\Gamma$ is contained in a hyper-plane.
\end{corollary}

\begin{proof}
  After a suitable rotation, we may assume that $\Gamma$ starts in
  $\partial \mathbb X \cap \mathcal H_d$.  Since the rotationally
  symmetric density $p_{\bm X}$ is reflectionally symmetric too, the
  assertions follows immediately form Lemma~\ref{lem:sym-dom}.
\end{proof}

\begin{example}
  If the density $p_{\bm X}$ on the cylinder
  $\mathcal C := \{\bm x : x_1^2 + x_2^2 \le r, x_3 \in [a,b]\}$ is
  rotationally symmetric, then every admissible principal curve
  starting at the boundary $\partial \mathcal C$ degenerates to a
  planar curve.
\end{example}

We call the density $p_{\bm X}$ with compact support
$\mathbb X = \mathbb B_r$ \emph{rotationally invariant}, if
$p_{\bm X}(\bm x) = p_{\bm X}(\bm y)$ for $\|\bm x\| = \|\bm y\|$,
where $\mathbb B_r$ denotes the ball of radius $r>0$.

\begin{corollary}
  Let $p_{\bm X}$ be rotationally invariant.  If the admissible
  principal curve $\Gamma$ starts in $\partial \mathbb B_r$, then
  $\Gamma$ is the straight line segment through the origin.
\end{corollary}

\begin{proof}
  Due to Corollary~\ref{cor:rot-sym}, $\Gamma$ is contained in a
  hyper-plane.  Since this holds true for every hyper-plane through
  $\Gamma(0)$ and the origin, the principal curve has to be a line
  segment.
\end{proof}

\subsection{Rectangular Triangles and Squares}

Next, we like to derive principal curves for uniform densities on
rectangular triangles in $\mathbb R^2$.  The two-dimensional special
case of our moving reference frame is just
\begin{equation*}
  T(\zeta) := 
  \begin{pmatrix}
    \cos \zeta \\
    \sin \zeta
  \end{pmatrix}
  \qquad\text{and}\qquad
  N(\zeta) :=
  \begin{pmatrix}
    -\sin \zeta \\
    \cos \zeta
  \end{pmatrix}
\end{equation*}
with $\zeta \in [0, 2 \pi)$.  We start by studying the system of
differential equations \eqref{eq:ode_system:2} for the infinite domain
$\mathbb X = \mathbb R^2_{\ge 0} := \{ (x_1,x_2) \in \mathbb R^2 :\,
x_1,x_2 \ge 0 \}$ equipped with the Lebesgue measure $\lambda$. Of
course $\lambda$ is not a probability measure on
$\mathbb R^2_{\ge 0}$.  However, the quantities $\mu_1$, $\mu_2$ in
\eqref{eq:trans-mom} are well defined up to the multiplicative factor
$1/V(\mathbb X)$, whenever $0 < \zeta < \pi/2$.  Here $V(\mathbb X)$
is the area of $\mathbb X$.  Since this factor cancels out in the
differential equations, we may think of $p_{\bm X} \equiv 1$.  For
particular initial conditions of the curve
$\Gamma:[0,\infty) \to \mathbb R^2_+$, we shall (numerically) find a
family of curves---admissable inside the interior of
$\mathbb R^2_+$---that oscillates slowly around the line
$\Gamma_0(s) = (s,s)$, $s \ge 0$.  Note, that $\Gamma_0$ is the most
obvious principal curve for the domain $\mathbb R^2_+$.

If $\Gamma(s) = (x_1(s), x_2(s))$ is admissible such that
$\zeta \in (0,\pi/2)$, then $u$ of the normal map
\eqref{eq:normal-map} lives in $\mathcal R(s) = [u_-(s), u_+(s)]$ with
\[
  u_{-}(s) := - x_2(s)/\cos(\zeta(s)), \qquad u_{+}(s) := x_1(s)/\sin(\zeta(s)).
\]
Based on the width
\[
  w(s)  := u_{+}(s) - u_{-}(s),
\]
we obtain
\[
  \begin{aligned}
    p_{\perp}(u,s) & = 
    \begin{cases}
      w(s)^{-1}, & u_{-}(s)  \le  u \le u_{+}(s),\\
      0, & \text{else},
    \end{cases}\\
    v_\perp(s) & = \tfrac12 (u_+(s) + u_-(s)),\\
    \var_\perp(s) & = \tfrac{1}{12} w(s)^2.
  \end{aligned}
\]
Note that the covariance matrix here reduces to the variance of the
transverse density.

Inserting the mean $v_\perp$ and the variance $\var_\perp$ into
\eqref{eq:ode_system:2}, we arrive at the following system
\begin{equation}
  \label{eq:PCdgl_triangle}
  \begin{aligned}
    \dot x_1(s) & = \cos(\zeta(s)),\\
    \dot x_2(s) & = \sin(\zeta(s)),\\
    \dot \zeta(s) & =  \frac32 \frac{ \frac{x_1(s)}{\sin(\zeta(s))} - \frac{x_2(s)}{\cos(\zeta(s))}}{\frac{x_1^2(s)}{\sin(\zeta(s))^2} - \frac{x_1(s) \, x_2(s)}{\sin(\zeta(s))\cos(\zeta(s))}+\frac{x_2^2(s)}{\cos(\zeta(s))^2}}.
  \end{aligned}
\end{equation}
In order to determine a principal curves, we like to start on the
$x_1$-axis, where the tangent vector is parallel to the $x_2$-axis,
i.e.,
\begin{equation}   \label{eq:initial_cond}
  x_1(0) > 0, \qquad x_2(0)=0, \qquad \zeta(0)  = \frac \pi 2.
\end{equation} 
Unfortunately, we cannot insert this initial conditions into
\eqref{eq:PCdgl_triangle}, since $x_2(s) / \cos(\zeta(s))$ is
undefined at these points.  However, by l'Hospital's rule, we can use
the continuous continuation as $s \to 0$ and incorporate the inital
conditions.  We use this observation and extend the system of
differential equations to
\begin{equation}
  \label{eq:PCdgl_triangle_extended}
  \begin{aligned}
    \dot x_1(s) & = \cos(\zeta(s)),\\
    \dot x_2(s) & = \sin(\zeta(s)),\\
    \dot \zeta(s) & =  \begin{cases} 
      \frac32 \frac{ \frac{x_1(s)}{\sin(\zeta(s))} - \frac{x_2(s)}{\cos(\zeta(s))}}{\frac{x_1^2(s)}{\sin(\zeta(s))^2} - \frac{x_1(s) \, x_2(s)}{\sin(\zeta(s))\cos(\zeta(s))}+\frac{x_2^2(s)}{\cos(\zeta(s))^2}}, & 0 < x_1,\; x_2,\; 0 < \zeta(s) < \frac \pi 2,\\
      \frac12 x_1^{-1}(s), & 0 < x_1,\; x_2 = 0,\; \zeta(s) = \frac \pi 2. 
    \end{cases}
  \end{aligned}
\end{equation}

Note that the system is homogeneous of degree $-1$ meaning that if
$\Gamma(s)$ is a solution, then for $t > 0$ the scaled version
$t^{-1} \Gamma(t s)$ is also a solution.  We solve this system
numerically using the method {\ttfamily odeint}
\cite{Hindmarsh1983,Petzold1983} in the Scipy-Python software, which
is based on the solver {\ttfamily lsoda} of the Fortran library
ODEPACK and is used in the remaining examples too.  The result is
shown in Figure~\ref{fig:solution_diff} left. Numerically we observe
the following: Let $\Gamma :[0, \ell] \to \mathbb R_{\ge 0}^2$ be a
solution of \eqref{eq:PCdgl_triangle_extended} with initial conditions
\eqref{eq:initial_cond}. Then for any $t>0$ the domain
\[
  \mathbb X_{t} := \{ \Gamma(s) + u N(\zeta(s)) : (s,u) \in (0,t)\times \mathbb R \} \cap \mathbb R^2_{\ge 0}  
\]
is a rectangular triangle and $\Gamma$ is a principal curve for the
uniform distribution on $\mathbb X_t$.  Moreover, the angle $\zeta(s)$
oscillates around $\frac \pi 4$, i.e, the function
$ \zeta(s) - \frac \pi 4$ has infinitely many zeros.  In particular,
this would imply that there are infinitely many closed principal
curves for the square, see Figure~\ref{fig:solution_diff} right, which
converge to the trivial non-smooth solution.

\begin{figure}
  \begin{center}
    \includegraphics[width=0.48\textwidth]{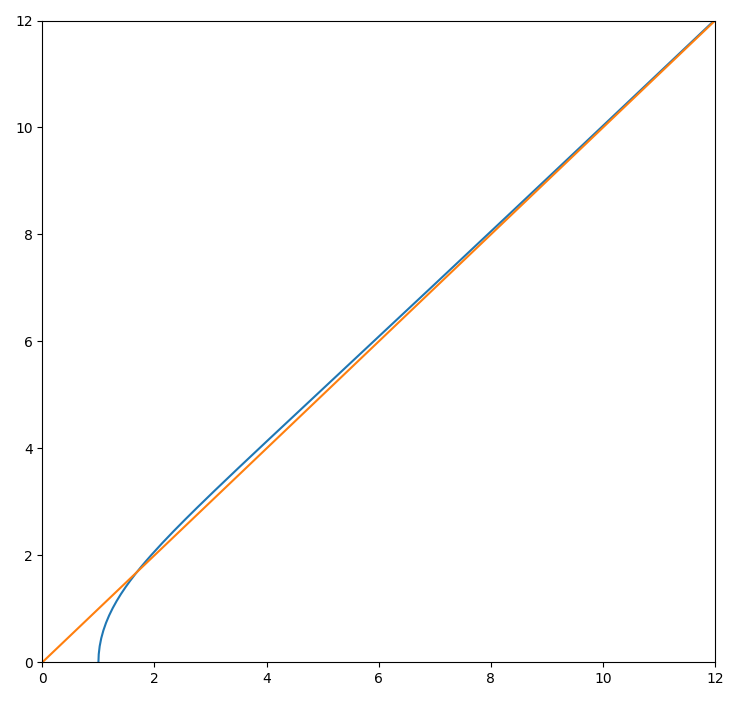}
    \includegraphics[width=0.48\textwidth]{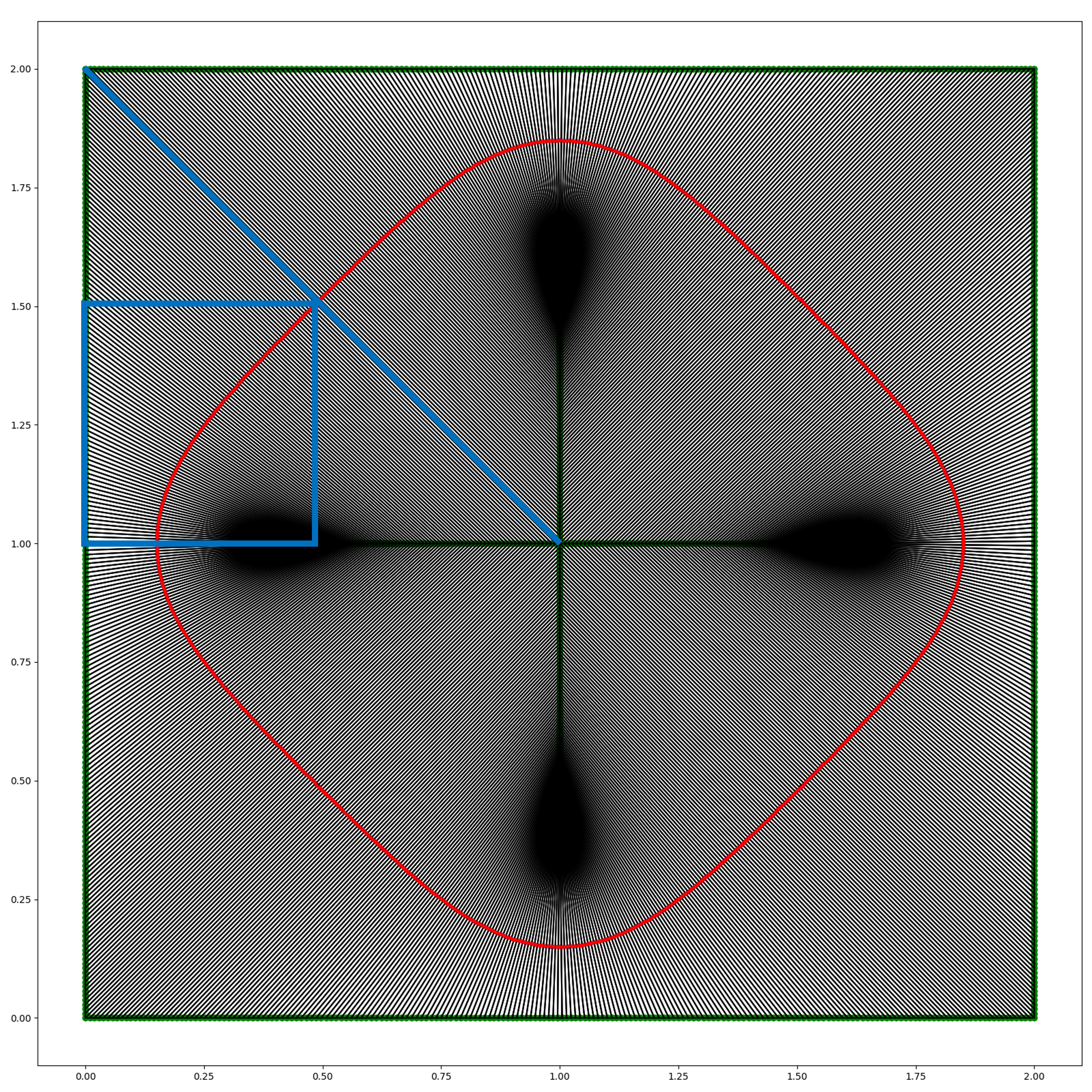}  
  \end{center}
  \caption{Left: A solution to \eqref{eq:PCdgl_triangle_extended} with
    $x_1(0)=1$, $x_2(0)=0$, $\zeta(0)=\frac \pi 2$ (blue) and for
    comparison the diagonal (orange).  Right: The derived non-trivial
    closed principal curve for the square (red) by composition of 8
    identical principal curves on rectangular triangles. We show the
    normals (black) and the border of the projection domain.}
  \label{fig:solution_diff}
\end{figure}

\subsection{Triangular-Based Prism}

Let $\Delta(E_1, E_2, E_3)$ be some triangle in the $x_1x_2$-plane,
and let
\begin{equation*}
  \mathcal P_\Delta :=
  \{ \bm x \in \mathbb R^3 : (x_1, x_2, 0) \in \Delta(E_1, E_2, E_3), x_3 \in [0, h]\}
\end{equation*}
be the corresponding prism of height $h$.  We want to compute a
principal curve starting at some point
$\Gamma(0) \in \Delta(E_1,E_2,E_3)$ with tangent
$\bm \zeta(0) = (\pi/2, \pi/2)$.  If $\Gamma$ is admissible, then the
cross-section between the prism and the normal planes at
$s \in (0,\ell_\Gamma)$ are not allowed to intersect with the two
bases of the prism.  Therefore, the cross-sections are again
triangles.  To compute the vertices of this triangles within the
normal coordinates, we may solve the equation systems
\begin{equation*}
  u_1 N_1(\bm \zeta(s)) + u_2 N_2(\bm \zeta(s)) - v \bm e_3 = E_i - \Gamma(s).
\end{equation*}
Notice that the system matrix $[N_1, N_2, -e_3]$ is triangular,
simplifying the computation of $u_1$ and $u_2$.  On the basis of these
vertices, we may split the integration over $u_1$ within the
definition of the transverse moments into integrals of the form
\begin{equation*}
  \tilde \mu_{(j,k)} (s) := \int_{v_1}^{v_2} \int_{a+bu_1}^{c+du_1} u_1^j u_2^k \, du_2 \, du_1.
\end{equation*}
The required first and seconds transverse moments are thus summations
about the partial moments
\begin{align*}
  \tilde \mu_{(1,0)}
  &= \frac{(d-b)}{3} \, (v_2^3 - v_1^3) + \frac{(c-a)}{2} \, (v_2^2 - v_1^2),
  \\
  \tilde \mu_{(2,0)}
  &= \frac{(d-b)}{4} \, (v_2^4 - v_1^4) + \frac{(c-a)}{3} \, (v_2^3 - v_1^3),
  \\
  \tilde \mu_{(1,1)}
  &= \frac{(d^2 - b^2)}{8} \, (v_2^4 - v_1^4) + \frac{(cd - ab)}{3} \, (v_2^3 - v_1^3) + \frac{(c^2 - a^2)}{4} 
    \, (v_2^2 - v_1^2),
  \\
  \tilde \mu_{(0,1)}
  &= \frac{(d^2 - b^2)}{6} \, (v_2^3 - v_1^3) + \frac{(cd-ab)}{2} \, (v_2^2 - v_1^2) + \frac{(c^2 - a^2)}{2} \, (v_2 - v_1),
  \\
  \tilde \mu_{(0,2)}
  &= \frac{(d^3 - b^3)}{12} \, (v_2^4 - v_1^4) + \frac{(cd^2 - ab^2)}{3} \, (v_2^3 - v_1^3) 
  \\
  &\qquad + \frac{(c^2d - a^2b)}{2} \, (v_2^2 - v_1^2) + \frac{(c^3 - a^3)}{3} \, (v_2 - v_1).
\end{align*}
Based on the moments, a principal curve of the triangular prism may be
computed by solving the differential equation system.  The results for
a specific triangle are shown in Figure~\ref{fig:tripri}.  The normal
planes do here not intersect so that the solution curve is admissible.
The computed curve coincides numerically with the curve through the
means of the sections $\pi_\Gamma^{-1}([t_k, t_{k+1}])$, where $t_k$
corresponds to the time steps of the solution curve; so the solutions
curve is self-consistent and hence a principal curve.  Using the above
procedure, we are able to compute principle curves of prism with
arbitrary polygonal base.

\begin{figure}
  \centering
  \includegraphics[width=0.9 \linewidth]{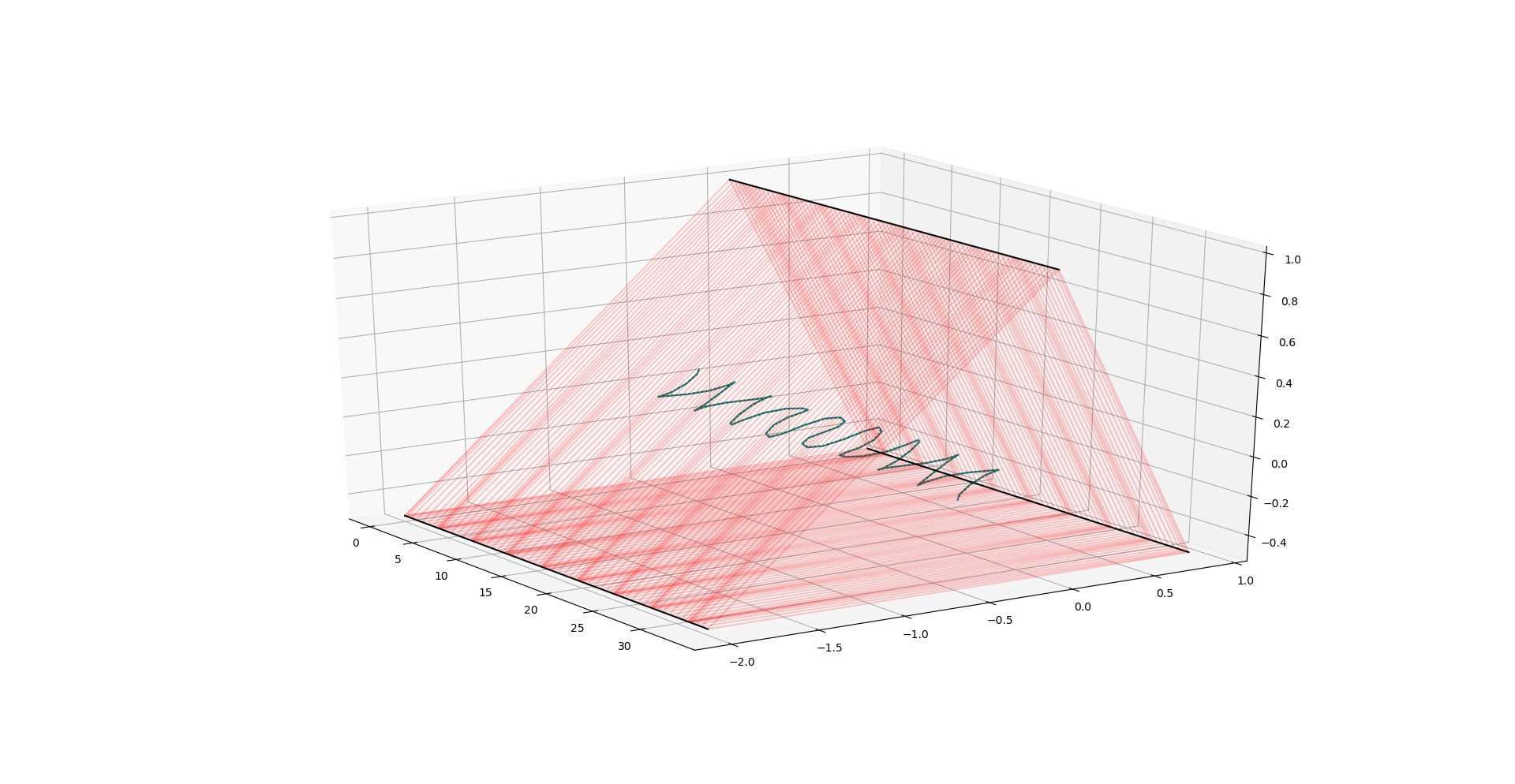}
  \caption{Principal curve in the triangular prism whose base
    corresponds to the points
    $(0,1,0)$, $(\sqrt3/2,-1/2,0)$, $(-2,-1/2,0)$.  The length of the
    prism is chosen such that both bases are parallel.  The
    intersection of the normal planes $\mathbb X(s)$ with the surface
    of the prism are shown as red triangles.  Since they do not
    intersect each other, the curve is admissible. The barycenters of
    the related Voronoi cells (green dots) numerically coincide with the
    curve.  Note the scaling with respect to the $x_3$-axis.}
  \label{fig:tripri}
\end{figure}

\subsection{Infinite Cylinder}

The arclength parameterization of a helix $\Gamma \equiv H$ is given
by
\[
  H(s) := 
  \begin{pmatrix}
    a \cos(ks) \\ a \sin(ks) \\ bks
  \end{pmatrix}, \qquad k:= 1/\sqrt{a^2+b^2}, \qquad a, b > 0.
\]
The corresponding Frenet frame reads as
\[
  T(s) =
  \begin{pmatrix}
    - ak \sin(ks) \\ ak \cos(ks) \\ bk    
  \end{pmatrix}, \qquad
  N(s) = 
  \begin{pmatrix}
    - \cos(ks) \\ - \sin(ks) \\ 0
  \end{pmatrix}, \qquad 
  B(s) =
  \begin{pmatrix}
    bk \sin(ks) \\ - bk \cos(ks) \\ ak
  \end{pmatrix}
\]
Moreover, it has constant curvature $\kappa$ and torsion $\tau$ given
by
\[
  \kappa = \frac{a}{a^2+b^2} = a k^2, \qquad \tau = \frac{b}{a^2+b^2} = b k^2.
\]

In what follows, we let 
\[
  \mathcal C_r := \{ (x_1, x_2, x_3)^\tT \;:\; x_1^2 + x_2^2 \le r^2 \}
\]
be the infinitely long cylinder of radius $r$.  To find appropriate
parameters, we consider \eqref{eq:int-cond}.  More precisely, we will
compute the integral
\[
  \begin{aligned}
    \bar H(s) & := A_s^{-1}\int_{\pi_{H}^{-1}(s)} (H(s) + u_1 N(s) + u_2 B(s)) (1-u_1\kappa) du_1 du_2,\\
    A_s & := \int_{\pi_{H}^{-1}(s)}(1 -u_1\kappa) du_1 du_2,
  \end{aligned}
\]
whose value has to coincide with $H(s)$ if the helix is a
principal curve, i.e.\ if \eqref{eq:int-cond} holds true.  Without
loss of generality we may assume $s=0$ so that
\[
  \bar H(0) = H(0) + \bar u_1 N(0) + \bar u_2 B(0) 
  = 
  \begin{pmatrix}
    a \\ 0 \\ 0   
  \end{pmatrix}
  + \bar u_1 
  \begin{pmatrix}
    -1 \\ 0 \\0 
  \end{pmatrix}
  + \bar u_2
  \begin{pmatrix}
    0 \\ -bk \\ ak
  \end{pmatrix}
\]
where 
\[
  \begin{aligned}
    \bar u_1 & := A_0^{-1}  \int_{\mathcal E_r} u_1 (1- \kappa u_1) du_1 du_2, \qquad
    \bar u_2 := A_0^{-1} \int_{\mathcal E_r} u_2 (1-\kappa u_1) du_1 du_2,\\
    \mathcal E_r & := \{ (u_1, u_2) : (a-u_1)^2 + (bku_2)^2 \le r^2 \}, \qquad
    A_0 :=  \int_{\mathcal E_r} (1-\kappa u_1) du_1 du_2.
  \end{aligned}
\]
Straightforward calculation leads to
\[
  \bar u_1= a\left(1 - \frac{r^2}{4b^2} \right) , \qquad \bar u_2=0
\]
such that
\[
  \bar H(0) = 
  \big(
    a r^2 / (4b^2) ,0 , 0
  \big)^{\mathrm T}.
\]
Setting $\bar H(0)$ equal to $H(0)$ and imposing the non-negativity of
the Jacobian determinant, i.e. $\kappa u_1 \le 1$,
$(u_1, u_2) \in \mathcal E_r$, we infer that the helix is a principal
curve for the uniform measure of the cylinder $\mathcal C_r$ if
\[
  b = r/2, \qquad 0 \le a \le r/4.
\]
Numerical experiments indicate that for $r/4 < a < 2r/3 $ there exists
$b < r/2$ such that the helix is also a principal curve, see
Figure~\ref{fig:helix}. Note the limiting case $b=0$ and $a=2/3$,
where the helix degenerates to a circle. However, in these cases the
helix has points of ambiguity inside the cylinder.

\begin{figure}
  \begin{center}
    \includegraphics[height=150pt]{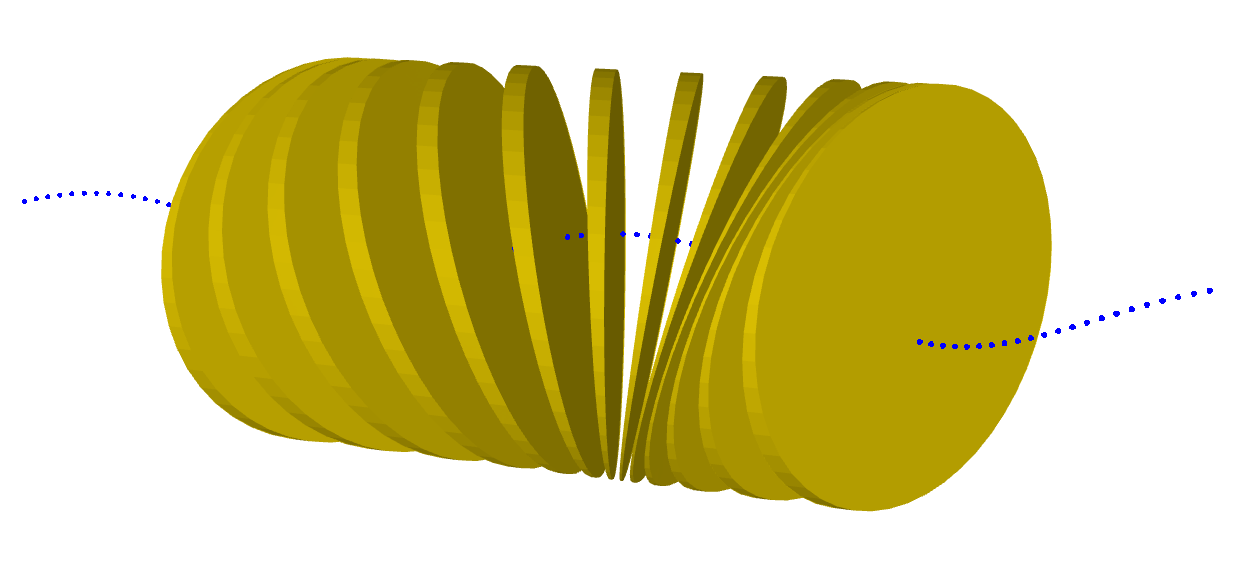}
    \includegraphics[height=150pt]{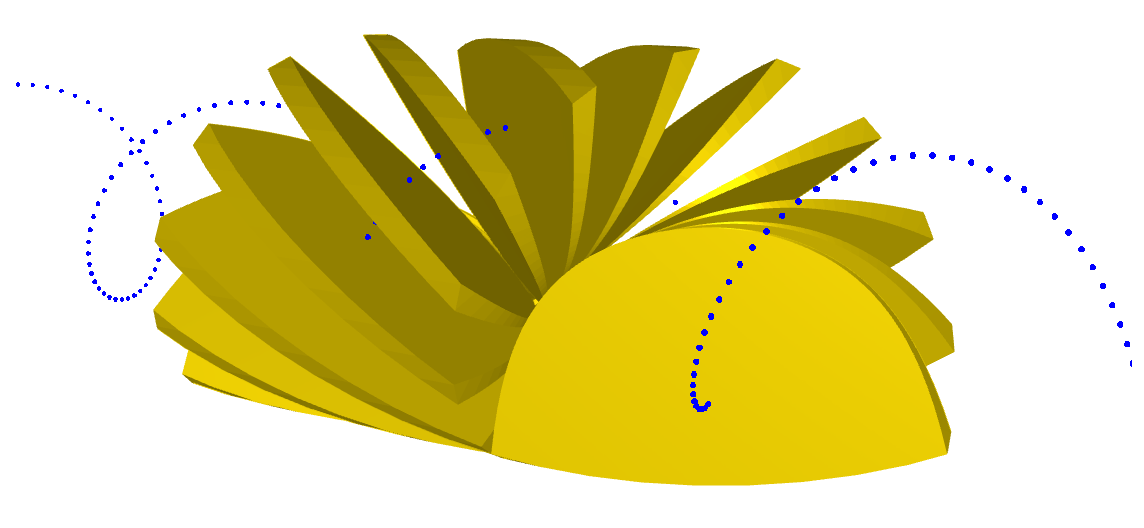}
    \includegraphics[height=150pt]{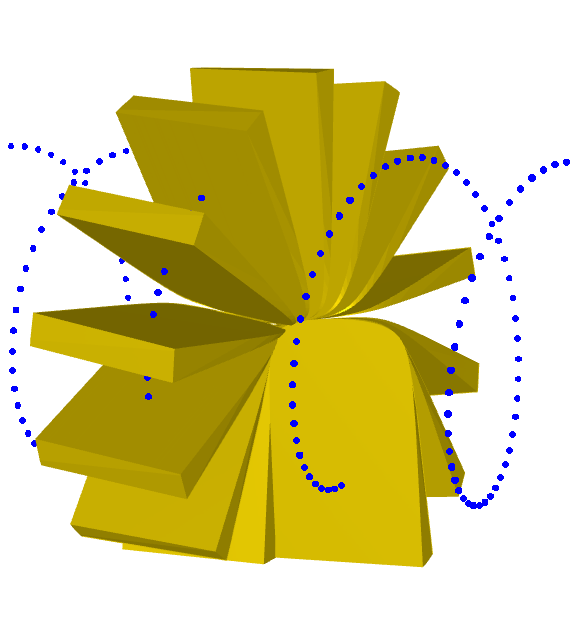}
  \end{center}
  \caption{Helices for several parameters
    $(a,b) \in \{(0.2, 0.5), (0.6, 0.35), (0.66, 0.1)\}$ (from top to
    bottom). The curves are sampled equidistantly (blue dots). For
    some sampling points the Voronoi regions intersected with the
    cylinder of radius $r=1$ are depicted in golden color.  Note that
    for the second and third set of parameters the helices are not
    admissible curves for the infinite cylinder.}
  \label{fig:helix}
\end{figure}

\section{Conclusion} \label{sec:conc}
We have derived a dynamical system for finding principal curves of
random variables in $\mathbb R^d$ for $d \ge 2$ and have numerically
computed the solution for uniformly distributed random variables with
density functions supported on certain domains.  It will be of
interest to consider also other distributions as, e.g. Gaussian
mixtures.  Another issue would be to have a look at principle curves
on manifolds as started in the papers \cite{Hauberg2016,KLO2020}.

Further, so far only the squared Euclidean norm was incorporated into
the considerations.  Unfortunately, classical PCA based on this
,,distance'' is sensitive to outliers so that robust methods were
considered in the literature, e.g.\ by skipping the square in the
Euclidean norm or taking the $\L_1$ norm.  For an overview of robust
subspace recovery, we refer to \cite{LM2018} and the references therein
and to recent results on robust principal lines \cite{NNSS19b}.  So
far we are not aware of a robust principal curve approach.
\\[2ex]

\textbf{Acknowledgement:} Funding by the DFG under Germany's
Excellence Strategy – The Berlin Mathematics Research Center MATH+
(EXC-2046/1, Projektnummer: 390685689) is acknowledged.

\bibliographystyle{abbrv}
\bibliography{literature}

\begin{thebibliography}{10}

\bibitem{BR1992}
J.~Banfield and A.~Raftery.
\newblock Ice floe identification in satellite images using mathematical
  morphology and clustering about principal curves.
\newblock {\em J. Am. Stat. Assoc.}, 87:7--16, 1992.

\bibitem{Bishop1975}
R.~L. Bishop.
\newblock There is more than one way to frame a curve.
\newblock {\em Am. Math. Mon.}, 82(3):246--251, 1975.

\bibitem{CG1998a}
K.-y. Chang and J.~Ghosh.
\newblock {Principal curves for nonlinear feature extraction and
  classification}.
\newblock {\em Appl. Artif. Neural Netw. Image Process. III}, 3307:120--129,
  1998.

\bibitem{CYYLP2016}
D.~Chen, J.~Yin, S.~Yang, L.~Li, and P.~Pudney.
\newblock {Constraint local principal curve: concept, algorithms and
  applications}.
\newblock {\em J. Comput. Appl. Math.}, 298:222--235, 2016.

\bibitem{GKL2019}
F.~de~Gournay, J.~Kahn, and L.~Lebrat.
\newblock Differentiation and regularity of semi-discrete optimal transport
  with respect to parameters of the discrete measure.
\newblock {\em Numer. Math.}, 141:429--453, 2019.

\bibitem{Del01}
P.~Delicado.
\newblock Another look at principal curves and surfaces.
\newblock {\em J. Multivar. Anal.}, 77(1):84--116, 2001.

\bibitem{DM1996}
D.~Dong and T.~J. McAvoy.
\newblock Nonlinear principal component analysis - based on principal curves
  and neural networks.
\newblock {\em Comput. Chem. Eng.}, 20(1):65--78, 1996.

\bibitem{DS93}
T.~Duchamp and W.~Stuetzle.
\newblock The geometry of principal curves in the plane.
\newblock Technical Report 250, Department of Statistics, GN-22, University of
  Washington, Seattle, February 1993.

\bibitem{DS1996}
T.~Duchamp and W.~Stuetzle.
\newblock Extremal properties of principial curves in the plane.
\newblock {\em Ann. Stat.}, 24(4):1520, 1996.

\bibitem{EGNSS2021}
M.~Ehler, M.~Gr{\"{a}}f, S.~Neumayer, and G.~Steidl.
\newblock Curve based approximation of measures on manifolds by discrepancy
  minimization.
\newblock {\em Found. Comput. Math.}, accepted.

\bibitem{Has84}
T.~Hastie.
\newblock Principal curves and surfaces.
\newblock Technical report, PhD Thesis, Stanford University, 1984.

\bibitem{HS89}
T.~Hastie and W.~Stuetzle.
\newblock Principal curves.
\newblock {\em J. Am. Stat. Assoc.}, 84(406):502--516, 1989.

\bibitem{Hauberg2016}
S.~Hauberg.
\newblock Principal curves on {R}iemannian manifolds.
\newblock {\em IEEE Trans. Pattern Anal. Mach. Intell.}, 38(9):1915--1921,
  2016.

\bibitem{Hindmarsh1983}
A.~C. Hindmarsh.
\newblock Odepack: A systematized collection of {ODE} solvers.
\newblock In {\em Sci. Comput.}, volume~1 of {\em IMACS Trans. Sci. Comput.},
  pages 55--64. North-Holland, Amsterdam, 1983.

\bibitem{HHM2010}
T.~Huckemann, S.~Hotz, and A.~Munk.
\newblock Intrinsic shape analysis: geodesic principle component analysis for
  {R}iemannian manifolds modulo {L}ie group actions.
\newblock {\em Stat Sin.}, 20:1--100, 2010.

\bibitem{Kegl1999}
B.~K\'egl.
\newblock Principal curves: learning, design, and applications.
\newblock {\em PhD thesis}, 1999.

\bibitem{KLO2020}
J.~H. Kim, J.~Lee, and H.~S. Oh.
\newblock {Spherical principal curves}.
\newblock {\em ArXiv Preprint}, 2003.02578, 2020.

\bibitem{Kle20}
A.~Klenke.
\newblock {\em Probability Theory}.
\newblock Universitext. Springer, Cham, 3rd edition, 2020.

\bibitem{KKLZ2000}
A.~Krzyzak, B.~K\'egl, T.~Linder, and K.~Zeger.
\newblock {Learning and Design of Principal Curves}.
\newblock {\em IEEE Trans. Pattern Anal. Mach. Intell.}, 22(3):281--297, 2000.

\bibitem{LGKW2018}
L.~Lebrat, F.~de~Gournay, J.~Kahn, and P.~Weiss.
\newblock Optimal transport approximation of 2-dimensional measures.
\newblock {\em SIAM J. Imaging Sci.}, 12(2):762--787, 2019.

\bibitem{LM2018}
G.~Lerman and T.~Maunu.
\newblock An overview of robust subspace recovery.
\newblock {\em Proc. IEEE}, 106(8):1380--1410, 2018.

\bibitem{NNSS19b}
S.~Neumayer, M.~Nimmer, S.~Setzer, and G.~Steidl.
\newblock On the robust {PCA} and {W}eiszfeld's algorithm.
\newblock {\em Appl. Math. Optim.}, 82:1017--1048, 2019.

\bibitem{P1901}
K.~Pearson.
\newblock On lines and planes of closest fit to systems of points in space.
\newblock {\em Philos. Mag.}, 2(11):559--572, 1901.

\bibitem{P2015}
X.~Pennec.
\newblock Barycentric subspaces and affine spans in manifolds.
\newblock {\em Int. Conf. Netw. Geom. Sci. Inform.}, pages 12--21, 2015.

\bibitem{Petzold1983}
L.~Petzold.
\newblock Automatic selection of methods for solving stiff and non-stiff
  systems of ordinary differential equations.
\newblock {\em SIAM J. Sci. Statist. Comput.}, 4(1):136--148, 1983.

\bibitem{RN1998}
K.~Reinhard and M.~Niranjan.
\newblock {Subspace Models For Speech Transitions Using Principal Curves}.
\newblock {\em Proc. Inst. Acoust.}, 20:53--60, 1998.

\bibitem{SLN2013}
S.~Sommer, F.~Lauze, and M.~Nielsen.
\newblock Optimization over geodesics for exact principle geodesic analysis.
\newblock {\em Adv. Comput. Math.}, 40:283--313, 2013.

\bibitem{fletcher2013}
P.~Thomas~Fletcher.
\newblock Geodesic regression and the theory of least squares on {R}iemannian
  manifold.
\newblock {\em Int. J. Comput. Vis.}, 105:171--185, 2013.

\bibitem{YILMAZ2010764}
S.~Yılmaz and M.~Turgut.
\newblock A new version of {B}ishop frame and an application to spherical
  images.
\newblock {\em J. Math. Anal. Appl.}, 371(2):764--776, 2010.

\end{thebibliography}

\end{document}